\documentclass[12pt]{amsart}

\usepackage{amssymb,bm,mathrsfs}

\usepackage{enumerate}

\usepackage{color}

\usepackage[centering, width=6in]{geometry}
\usepackage[colorlinks,pagebackref]{hyperref}

\makeatletter
\@namedef{subjclassname@2020}{\textup{2020} Mathematics Subject Classification}
\makeatother

\theoremstyle{plain}
\newtheorem{thm}{Theorem}[section]
\newtheorem*{thm*}{Theorem}
\newtheorem{prop}{Proposition}[section]
\newtheorem{lem}{Lemma}[section]

\newtheorem{conj}{Conjecture}[section]
\theoremstyle{definition}
\newtheorem{defn}{Definition}[section]

\theoremstyle{remark}
\newtheorem{rem}{Remark}[section]
\numberwithin{equation}{section}

\DeclareMathOperator{\ke}{Ker}

\newcommand{\N}{\mathbb N}

\newcommand{\Z}{\mathbb Z}

\begin{document}
	
	\title{Non-Wieferich property of prime ideals and a conjecture of Erd\"os}

	\author{Ruofan Li}
	
	\address{Department of Mathematics, Jinan University, Guangzhou, 510632, China}
	
	\email{liruofan@jnu.edu.cn}

	\author{Jiuzhou Zhao$^*$}
	
	\address{School of Mathematics and Statistics, Key Laboratory of Engineering Modeling and Statistical Computation of Hainan Province, Hainan University, Haikou 570228, China}
	
	\email{zhao9zone@gmail.com}

	
	
	\thanks{$^*$Corresponding author.}
	
	
	
	\begin{abstract}
		Let $K$ be a number field with ring of integers $\mathcal{O}$ and $\alpha\in\mathcal{O}$. For any prime ideal $\mathfrak{p}$ of $\mathcal{O}$, we obtain its higher $\alpha$-Wieferich property, which implies a nonexistence theorem for higher Wieferich unramified prime ideals. If $\beta\in\mathcal{O}$ is relatively prime to $\alpha$ and all prime ideal factors of $(\beta)$ are unramified and have residue degree $1$, we apply our higher $\alpha$-Wieferich property to establish the asymptotic equidistribution of digits in $\beta$-adic expansions of $\alpha^n$, which is a generalization of the Dupuy-Weirich theorem. When $(\beta)$ have ramified prime ideal factors, we also  obtain a result on the block complexity of $\beta$-adic expansions of $\alpha^n$.
	\end{abstract}
	
	\maketitle
	
	\section{Introduction}\label{sc:intro}
	\subsection{A conjecture of Erd\"os and Dupuy-Weirich theorem}\label{2512091258}
	For positive integers $m$ and $q$, if $$m = \sum_{i=0}^{k} a_{i} q^{i}$$ for some integers $a_{i}$ between $0$ and $q-1$, then we write $(m)_q = a_{k} \cdots a_{0}$ and call it the $q$-ary expansion of $m$. An interesting fact is that, other than $(2^0)_3=1$, $(2^2)_{3}=11$ and $(2^8)_3=100111$, no other values of n such that $(2^n)_3$ omits the digit $2$ is known. In \cite{Erdos1979}, Erd\"os  proposed the following conjecture:
	\begin{conj}\label{2601051536}
		There are only finitely many powers of $2$ whose 
		ternary expansion omits the digit $2$.
	\end{conj}
	Most progress towards Conjecture \ref{2601051536} has been in the form of upper bounds on
	\[M(N):=\#\{n\in\Z\cap[1,N]\colon (2^n)_3\text{ omits the digit $2$}\},\]
	where the symbol $\#$ denote cardinality. The best known bound on $M(N)$ is due to Narkiewicz \cite{Narkiewicz1980} who proved that
	\begin{equation}\label{2402121459}
		M(N)\le1.62N^{\sigma}, \quad\text{	where $\sigma:=\log_32\approx0.63092$. }
	\end{equation}
	We refer the reader to \cite{Dupuy2016a,Holdum2015,Kennedy2000,Lagarias2009,LZ2025} for more results related to Narkiewicz's result. Recently, Dimitrov and Howe \cite[Theorem 1.2]{DH2025} have demonstrated that if $n\notin\{0,2,8\}$, then the ternary expansion of $2^n$ either contains the digit $2$ or includes at least twenty-six $1$'s. 
	
	Based on computer experiments, Dupuy and Weirich \cite{Dupuy2016a} believe that a stronger version of Conjecture \ref{2601051536} holds.
	\begin{conj}\label{2504301154}
		Let $p$ and $q$ be distinct primes and $b\in\{0,1,\dots,q-1\}$. Denote by $d_n(b)$ the number of $b$'s appearing in $(p^n)_q$. Then
		\begin{equation}\label{2504300835}
			\lim_{n\to\infty}\frac{d_n(b)}{n\log_q(p)}=\frac{1}{q}
		\end{equation}
	\end{conj}
	\begin{rem}
		(i)  Note that the denominator $n\log_q(p)$ in \eqref{2504300835} is the length of the $q$-ary expansion of $p^n$, this limit means that the proportion of digits $b$'s in $(p^n)_q$ tends to $1/q$ as $n\to\infty$.\\
		(ii) Take $p=2$ and $q=3$, \eqref{2504300835} implies that $d_n(2)>0$ for $n$ large enough, which implies Conjecture \ref{2601051536}.
	\end{rem}
	
	The following theorem is the first progress on Conjecture \ref{2504301154}, which considers the average proportion of $b$'s in $(p^n)_q$. Given the $q$-ary expansion
	\[p^n=a_0+a_1q+\cdots+a_Nq^N,\]
	where $N=\lfloor\log_q(p^n)\rfloor$ and $a_i\in\{0,1,\dots,q-1\}$. Let $(p^n)_{q,i}:=a_i$. Denote
	\[f_{p,n,m}(b):=\frac{\#\{i\in\Z\cap[0,m-1]\colon (p^n)_{q,i}=b\}}{m},\]
	for $b\in\{0,1,\dots,q-1\}$. Let $L_m:=\{p^n+q^m\Z\colon n\in\Z\}\subset(\Z/q^m\Z)^\times$ and $l_m:=\#L_m$. It is clear that for each $m\in\Z_{\ge1}$, the average proportion of $b$'s in the first $m$ digits satisfies
	\begin{equation}\label{2504301428}
		f_{p,m}(b):=\lim_{N\to\infty}\frac{1}{N}\sum_{n=1}^{N}f_{p,n,m}(b)=\frac{1}{l_m}\sum_{n=1}^{l_m}f_{p,n,m}(b).
	\end{equation}
	\begin{thm*}[Dupuy and Weirich~{\cite[Theorem~3]{Dupuy2016a}}]
		Let $p$ and $q$ be distinct primes and $b\in\{0,1,\dots,q-1\}$, then
		\[\lim_{m\to\infty}f_{p,m}(b)=1/q.\]
	\end{thm*}
	
	In this paper, we are going to generalize the above Dupuy-Weirich theorem to general number fields. For this purpose, we also generalize their work \cite[Theorem~6]{Dupuy2016a} on Wieferich primes. It is worth mentioning that we have observed new phenomena in the setting of general number fields.  
	
	\subsection{Wieferich primes}
	Let $p$ be a rational prime. For any $k\in\Z_{\ge1}$, we denote \[\langle2\rangle_{p^k}:=\{2^n\!\!\pmod{p^k}\colon n\in\Z_{\ge1}\}\]
	as the multiplicative group generated by $2$ modulo $p^k$. 
	The prime $p$ is called \emph{(classical) Wieferich} if $\langle2\rangle_{p}$ is isomorphic to $\langle2\rangle_{p^2}$. 
	
	Such primes were initially investigated in \cite{Wie09} in relation to the Fermat's Last Theorem. In that work, the author demonstrated that if \( x^q + y^q = z^q \) forms a Fermat triple, then \( q \) must be a Wieferich prime. Whether there exist infinitely many Wieferich primes remains an open problem, even assuming the validity of the ABC conjecture. For an extensive discussion of Wieferich primes, we refer the reader to \cite{Crandall,Katz2015,Lang1990,Silverman1988}. 
	
	From now on, fix a number field $K$ and let  $\mathcal{O}$ be its ring of integers. 
	\begin{defn}
		Let $\mathfrak{p}$ be a prime ideal of the ring $\mathcal{O}$, $r\in\Z_{\ge1}$ and $\alpha\in\mathcal{O}\backslash\mathfrak{p}$. We say that the prime ideal $\mathfrak{p}$ is \emph{$\alpha$-Wieferich} at $r$ if the multiplicative group generated by $\alpha$ modulo $\mathfrak{p}^r$ is isomorphic to the multiplicative group generated by $\alpha$ modulo $\mathfrak{p}^{r-1}$.
	\end{defn}
	
	We now establish a higher $\alpha$-Wieferich property of prime ideals when the integer $r$ is large enough. Let  $G_k=\langle\alpha\rangle_{\mathfrak{p}^k}=\{\alpha^n\!\!\pmod{\mathfrak{p}^k}\colon n\in\Z_{\ge1}\}$ be the multiplicative group generated by $\alpha$ modulo $\mathfrak{p}^k$. There is a natural homomorphism from $G_r$ to $G_{r-1}$, and we denote the kernel of it by $\ke(G_r\to G_{r-1})$. It is clear that $\mathfrak{p}$ is $\alpha$-Wieferich at $r$ if and only if $\#\ke(G_r\to G_{r-1})=1$.
	\begin{thm}\label{2505091458}
		Let $\mathfrak{p}$ be a prime ideal of the ring $\mathcal{O}$. Assume that $\alpha\in\mathcal{O}\backslash\mathfrak{p}$ is not a root of unity, then there is a positive integer $v$ such that
		\[\#\ke(G_r\to G_{r-1})=\begin{cases}
			1,&\quad \text{ if } r-v\not\equiv1\pmod{e}\\
			p,&\quad \text{ if } r-v\equiv1\pmod{e}
		\end{cases}\]
		holds for all $r> v$, where $e$ is the ramification index of $\mathfrak{p}$ and $p$ is the rational prime lying below $\mathfrak{p}$. In particular, if $\mathfrak{p}$ is unramified, then $\mathfrak{p}$ is not $\alpha$-Wieferich at $r$ for all $r>v$. 
	\end{thm}
	
	\begin{rem}
		The condition that $\alpha$ is not a root of unity is necessary. If $\alpha$ is a $m$-th root of unity, then $\langle \alpha\rangle=\{\alpha,\alpha^2,\dots,\alpha^m=1\}\subseteq\mathcal{O}$ is a finite set. Take 
		$$N=\max\Big\{l\in\Z_{\ge1}\colon  \alpha^i-\alpha^j\in\mathfrak{p}^{l}\backslash\mathfrak{p}^{l+1}\text{ for some }i,j\in\{1,2,\dots,m\}\Big\}.$$
		For any $r>N$, the multiplicative group generated by $\alpha$ modulo $\mathfrak{p}^{r}$ is isomorphic to $\{\alpha,\alpha^2,\dots,\alpha^m\}$. Hence $\mathfrak{p}$ is $\alpha$-Wieferich at $r$ for all $r>N+1$. 
	\end{rem}

	\subsection{Asymptotic equidistribution of digits}
	Theorem \ref{2505091458} enables us to generalize the Dupuy-Weirich \cite{Dupuy2016a} theorem to general number fields. 
	
	We first recall the concept of \emph{$\beta$-adic expansions}, which is a natural generalization of $q$-ary expansions. Denote $N(\beta):=\#(\mathcal{O}/\beta\mathcal{O})$ for every  nonzero $\beta\in \mathcal O$.  
	\begin{defn}
		Fix $\beta\in \mathcal O$ with norm $N(\beta)>1$ and a set of representatives $\mathcal D$ of the quotient group $\mathcal{O} / \beta \mathcal{O}$ with $0\in\mathcal D$. For every $x \in \mathcal{O}$, the \emph{$\beta$-adic expansion} of $x$ (with respect to $\mathcal D$) is the unique sequence $(b_i)_{i\in\N} \in \mathcal D^{\N}$ such that 
		\begin{equation}\label{2401171345}
			x=\lim_{i\rightarrow\infty}b_0+b_1\beta+\cdots+b_i\beta^i
		\end{equation}
		with respect to $\mathfrak{p}$-adic topology for any prime ideal $\mathfrak{p}$ in $\mathcal{O}$ dividing $\beta$. 
	\end{defn}
	The existence and uniqueness of $\beta$-adic expansions were proved in \cite[Section 2]{LZ2025}. 
	
	We define the \emph{$i$-th digit function} $(x)_{\beta,i}:=b_i$ for $x\in\mathcal{O}$ with $x=\sum_{i=0}^{\infty}b_i\beta^i$. Given an algebraic integer $\alpha\in\mathcal{O}$, we denote the frequency of digit $b\in\mathcal D$ in the $m$-truncated expansion of $\alpha^n$ by
	\begin{equation}\label{2601061743}
		f_{\alpha,n,m}(b):=\frac{\#\{i\colon 0\le i<m,\;(\alpha^n)_{\beta,i}=b\}}{m}. 
	\end{equation}
	Consider the Ces\`aro average $\frac{1}{N}\sum_{n=1}^{N}f_{\alpha,n,m}(b)$, it is clear that
	\begin{equation}\label{2601061744}
		f_{\alpha,m}(b):=\lim_{N\to\infty}\frac{1}{N}\sum_{n=1}^{N}f_{\alpha,n,m}(b)=\frac{1}{h_m}\sum_{n=1}^{h_m}f_{\alpha,n,m}(b),
	\end{equation}
	where $h_m:=\#H_m$ and $H_m:=\langle\alpha+\beta^m\mathcal{O}\rangle\subset(\mathcal{O}/\beta^m\mathcal{O})^\times$. 
	
	We establish the asymptotic equidistribution of digits. Denote $N(\mathfrak{a}):=\#(\mathcal{O}/\mathfrak{a})$ for any ideal $\mathfrak{a}\subseteq\mathcal{O}$.
	\begin{thm}\label{2505041026}
		Let $(\beta)=\mathfrak{p}_1^{g_1}\cdots\mathfrak{p}_h^{g_h}$ satisfy that $\mathfrak{p}_i$ is unramified and \(N(\mathfrak{p}_i) = p_i\) for all $i$, where $p_i$ is the rational prime lying below $\mathfrak{p}_i$. If $\alpha$ is relatively prime to $\beta$ and is not a root of unity, then
		\begin{equation}\label{2512251734}
			\lim_{m\to\infty}f_{\alpha,m}(b)=1/\#\mathcal{D}
		\end{equation}
		holds for all $b\in\mathcal{D}$.
	\end{thm}
	
	When the prime ideal factorization of $(\beta)$ contains a ramified prime ideal $\mathfrak{p}$, despite that its $\alpha$-Wieferich property (as in Theorem \ref{2505091458}) is not sufficient for the proof of the asymptotic equidistribution of digits, we are still able to obtain a result on the block complexity of $\beta$-adic expansions of $\alpha^n$. 
	
	Denote the number of $m$-truncated blocks by
	\[\mathcal{C}_m(\alpha):=\#\Big\{(b_0,\dots,b_{m-1})\in\mathcal{D}^m\colon\alpha^n\equiv\sum_{i=0}^{m-1}b_i\beta^i \!\!\pmod{\beta^m\mathcal{O}}\;\;\text{for some $n$}\Big\},\]
	we define the block complexity as 
	\[\mathcal{C}(\alpha):=\lim_{m\to\infty}\frac{\log\mathcal{C}_m(\alpha)}{m\log N(\beta)},\quad\text{if the limit exists}.\]
	\begin{thm}\label{2512261042}
		Let $(\beta)=\mathfrak{p}_1^{g_1}\cdots\mathfrak{p}_h^{g_h}$.  If $\alpha$ is not a root of unity and relatively prime to $\beta$, then
		\begin{equation}\label{2512261352}
			\mathcal{C}(\alpha)=\frac{\sum_{j=1}^{h}g_je_j^{-1}\log p_j}{\sum_{j=1}^{h}g_jf_j\log p_j},
		\end{equation}
		where $p_j$ is the rational prime lying below $\mathfrak{p}_j$, $e_j$ is the ramification index of $\mathfrak{ p}_j$ and $f_j$ is the residue degree of $\mathfrak{ p}_j$.
	\end{thm}

	\section{The valuation ring and Teichmuller map}
	
	Note that Theorem \ref{2505091458} concerns only a single prime ideal $\mathfrak{p}$, so it is convenient to consider the localization of \( \mathcal{O} \) at \( \mathfrak{p}\). Let \( K_{\mathfrak{p}} \) be the completion of the number field \( K \) with respect to the \( \mathfrak{p} \)-adic valuation \( v_\mathfrak{p} \). The \emph{\(\mathfrak{p} \)-adic integer ring}
	\[\mathcal{O}_\mathfrak{p} := \{ x \in K_\mathfrak{p} : v_\mathfrak{p}(x) \geq 0 \}\]
	contains the \emph{group of units}
	\[\mathcal{O}_\mathfrak{p}^\times := \{ x \in K_\mathfrak{p} : v_\mathfrak{p}(x) = 0 \}\]
	and the \emph{unique maximal ideal}
	\[\mathcal{P} := \{ x \in K_\mathfrak{p} : v_\mathfrak{p}(x) > 0 \}.\]
	
	Note that 
	$\mathcal{O}/\mathfrak{p}^k \cong \mathcal{O}_\mathfrak{p}/\mathcal{P}^k$ 
	holds for any \( k \in \mathbb{Z}_{\geq 1} \). 
	For convenience, we still denote the cyclic subgroup 
	$\langle \alpha + \mathcal{P}^k \rangle \subseteq (\mathcal{O}_\mathfrak{p}/\mathcal P^k)^\times$
	as \( G_k \).
	
	Before studying \( \ke(G_k \to G_{k-1}) \), we need to recall some basic facts about \( \mathcal{O}_\mathfrak{p} \). Fix a prime element \( \pi \in \mathcal{O}_\mathfrak{p} \), i.e. an element satisfying $ v_\mathfrak{p}(\pi) = 1$, we have the following properties (see \cite[Chapter II, Section 3 and 4]{Neu1992} for details).
	\begin{prop}\label{2512231208}
		(i) Any \( x \in \mathcal{O}_\mathfrak{p} \) has a prime factorization
		\[ x = \varepsilon \pi^i, \]
		where \( \varepsilon \in \mathcal{O}_\mathfrak{p}^{\times} \) is a unit and \( i = v_\mathfrak{p}(x) \). 
		
		\noindent(ii) The nonzero proper ideals in \( \mathcal{O}_\mathfrak{p}\) are the following principal ideals:
		\[ \mathcal{P} = (\pi), \; \mathcal{P}^2 = (\pi^2), \; \ldots, \; \mathcal{P}^k = (\pi^k), \; \ldots \]
		
		\noindent(iii) Fix a set of representatives \( R \) for \( \mathcal{O}_\mathfrak{p}/\mathcal{P} \), containing \( 0 \). Then any \( x \in \mathcal{O}_\mathfrak{p} \) has a \( \mathfrak{p} \)-adic expansion
		\begin{equation*}
			x = a_0 + a_1 \pi + a_2 \pi^2 + \cdots + a_k \pi^k + \cdots 
		\end{equation*}
		where \( a_k \in R \) for all \( k \in \mathbb{Z}_{\geq 0} \).
	\end{prop}
	
	The following lemma is rather elementary but plays a key role in our proof. 
	Recall that $p$ is the rational prime lying below $\mathfrak{p}$, and $e$ is the ramification index of $\mathfrak{p}$.
	
	\begin{lem}\label{2601061529}
		Let $a,b\in\mathcal{O}_{\mathfrak{p}}^{\times}$, if 
		$a\equiv b\!\!\pmod{\mathcal{P}^v}$ but $a\not\equiv b\!\!\pmod{\mathcal{P}^{v+1}}$ for some $v\in\Z_{\ge1}$, then 
		\begin{equation}\label{25112141150}
			a^p\equiv b^p\!\!\pmod{\mathcal{P}^{\min\{e+v,pv\}}}. 
		\end{equation}
		Moreover, if $pv>e+v$, then
		\begin{equation}\label{2512141206}
			a^p\equiv b^p\!\!\pmod{\mathcal{P}^{e+v}}\;\;\text{but }\;a^p\not\equiv b^p\!\!\pmod{\mathcal{P}^{e+v+1}}.
		\end{equation}
	\end{lem}
	
	\begin{proof}
		Since $a\equiv b\!\!\pmod{\mathcal{P}^v}$ but $a\not\equiv b\!\!\pmod{\mathcal{P}^{v+1}}$, we have \( a = b + \varepsilon \pi^{v} \) for some $\varepsilon\in\mathcal{O}_\mathfrak{p}^{\times}$. Raise both sides to the $p$-th power, we obtain that
		\begin{equation}\label{2512141153}
			\begin{split}
				a^{p}-b^p& = (b  + \varepsilon \pi^{v})^p-b^p \\
				&= \binom{p}{1} b^{p-1}\varepsilon \pi^{v} + \binom{p}{2}b^{p-2} \varepsilon^2 \pi^{2v} + \cdots + \varepsilon^p \pi^{pv}. 
			\end{split} 
		\end{equation}
		Note that \( \binom{p}{j} \) is divisible by \( p \) but not \( p^2 \)  for all \( 1 \leq j \leq p-1 \), so
		\[
		v_\mathfrak{p} \left( \binom{p}{j}b^{p-j} \varepsilon^j \pi^{j v} \right) = e + jv,\quad \text{for $j\neq p$}.
		\]
		Hence
		\(
		v_\mathfrak{p} (a^p-b^p) \geq \min\{ e + v, p v \}
		\), which implies \eqref{25112141150}. 
		
		Moreover, if $pv>e+v$, $v_\mathfrak{p} \left( \binom{p}{j}b^{p-j} \varepsilon^j \pi^{j v} \right) = e + v$ only when $j=1$, thus \eqref{2512141206} holds. 
	\end{proof}
	
	We conclude this section by defining the Teichmuller map in the standard way. 
	\begin{defn}
		The Teichmuller map $\tau\colon\mathcal{O}_{\mathfrak{p}}^{\times}\to\mathcal{O}_{\mathfrak{p}}^{\times}$ is defined by
		\begin{equation}\label{2505091040}
			\tau(x)=\lim_{n\to\infty}x^{p^{fn}},
		\end{equation}
		where  $f$ is the residual degree of $\mathfrak{p}$ and $p$ is the rational prime lying below $\mathfrak{p}$. 
	\end{defn}
	
	Recall that the $\mathfrak{ p}$-adic absolute value is defined as $|x|_\mathfrak{p}:=c^{-v_{\mathfrak{p}}(x)}$ for $x\in K_\mathfrak{ p}$, where $c$ is some constant strictly greater than $1$.
	\begin{lem}\label{25050091445}     	
		For any $x \in \mathcal{O}_{\mathfrak{p}}^{\times}$, $\tau(x)$ is well-defined and we have 
		\begin{enumerate}[\upshape(i)]
			\item\label{2505091447}  $\tau(x)$ satisfies the equation $\big(\tau(x)\big)^{p^f}=\tau(x)$;
			\item\label{2505091446}  we have $\tau(x)\equiv x\pmod{\mathcal{P}}$. 
		\end{enumerate}
	\end{lem}
	
	\begin{proof}
		We first show that the sequence $(x^{p^{fn}})_{n\ge1}$ is a Cauchy sequence with respect to the $\mathfrak{p}$-adic metric, hence the limit \eqref{2505091040} exists.
		
		For each $x\in\mathcal{O}_{\mathfrak{p}}^{\times}$, since $\#(\mathcal{O}_{\mathfrak{p}}/\mathcal{P})^\times=p^f-1$,  we have $x^{p^f-1}\equiv1\pmod{\mathcal{P}}$.
		Apply Lemma \ref{2601061529} with $a=x^{p^f-1}$ and $b=1$, we have $x^{(p^f-1)p}\equiv1\pmod{\mathcal{P}^{2}}$. Apply Lemma \ref{2601061529} repeatedly, we obtain that
		\begin{equation}\label{2505091439}
			x^{(p^f-1)p^{fm}}\equiv 1\!\!\pmod{\mathcal{P}^{1+fm}}\quad\text{for all $m\in\Z_{\ge0}$}.
		\end{equation}
		Thus $\big|x^{p^{f(m+1)}}- x^{p^{fm}}\big|_{\mathfrak{p}}=\big|x^{p^{fm}}\big|_{\mathfrak{p}}\cdot\big|x^{(p^f-1)p^{fm}}-1\big|_{\mathfrak{p}}\le c^{-1-fm}$ for all $m\in\Z_{\ge0}$, where $|\cdot|_\mathfrak{p}:=c^{-v_{\mathfrak{p}}(\cdot)}$ for some constant $c>1$. Using this inequality, we obtain that for any integers $n>m\ge1$, 
		\[\begin{split}
			|x^{p^{fn}}- x^{p^{fm}}|_{\mathfrak{p}}&\le\max\big\{|x^{p^{fn}}- x^{p^{f(n-1)}}|_{\mathfrak{p}},\dots,|x^{p^{f(m+1)}}- x^{p^{fm}}|_{\mathfrak{p}}\big\}\\
			&\le c^{-1-fm}\to0\;\;(\text{as } m\to\infty),
		\end{split}\]
		which shows that the sequence $(x^{p^{fn}})_{n\ge1}$ is a Cauchy sequence.
		
		For all $m\in\Z_{\ge0}$, multiply both sides of \eqref{2505091439} by $x^{p^{fm}}$, we have
		\[x^{p^{f(m+1)}}\equiv x^{p^{fm}}\pmod{\mathcal{P}^{1+fm}}.\]
		Hence
		\[x^{p^{fm}}\equiv x^{p^{f(m-1)}}\equiv\cdots\equiv x\pmod{\mathcal{P}}\]
		for all $m\in\Z_{\ge1}$. Letting $m\to\infty$, we obtain Lemma \ref{25050091445} \eqref{2505091446}. 
		
		To see Lemma \ref{25050091445} \eqref{2505091447}, we note that $y\mapsto y^{p^f}$ is a continuous map on $K_{\mathfrak{p}}$, so
		\[(\tau(x))^{p^f}=\big(\lim_{n\to\infty}x^{p^{fn}}\big)^{p^f}=\lim_{n\to\infty}x^{p^{f(n+1)}}=\tau(x).\qedhere\]
	\end{proof}
	
	\section{Proof of Theorem \ref{2505091458}}
	
	Let us start with the direct product decomposition of \((\mathcal{O}_\mathfrak{p} / \mathcal{P}^k)^{\times}\) for each $k\in\Z_{\ge1}$. Denote the multiplicative group of principal units modulo $\mathcal{P}^k$
	\[
	\{ x+\mathcal{P}^k  : x \equiv 1 \!\!\pmod{\mathcal{P}} \}
	\]
	as \(U_1^{(k)}\). Consider the following short exact sequence
	\[
	1 \longrightarrow U_1^{(k)} \longrightarrow (\mathcal{O}_\mathfrak{p} / \mathcal{P}^k)^{\times} \xrightarrow[\quad\;\;]{\varphi_k} (\mathcal{O}_\mathfrak{p} /\mathcal{ P})^{\times} \longrightarrow 1,
	\]
	where \( \varphi_k: x \bmod \mathcal{P}^k \mapsto x \bmod \mathcal{P} \). Note that this short exact sequence splits, because there exists a homomorphism 
	\[
	\begin{split}
		\psi_k: (\mathcal{O}_\mathfrak{p} /\mathcal{ P})^{\times} &\longrightarrow\;\;(\mathcal{O}_\mathfrak{p} / \mathcal{P}^k)^{\times}\\
		a+\mathcal{P}\quad&\longmapsto \;\;\tau(a)+\mathcal{P}^k,
	\end{split}
	\]
	where \( a \in R \) (the set of representatives as in Proposition \ref{2512231208} (iii)) and \(\tau\) is the Teichmüller map; and one can check that
	\[
	\varphi_k \circ \psi_k = \operatorname{id}_{(\mathcal{O}_\mathfrak{p} /\mathcal{ P})^{\times}}.
	\]
	Hence
	\begin{equation}\label{2512231239}
		(\mathcal{O}_\mathfrak{p} / \mathcal{P}^k)^{\times} \cong (\mathcal{O}_\mathfrak{p} /\mathcal{ P})^{\times} \times U_1^{(k)}.
	\end{equation}
	
	Suppose \(\alpha \in (\mathcal{O}_\mathfrak{p})^{\times}\) is not a root of unity and its \(\mathfrak{p}\)-adic expansion is
	\[ \alpha = a_0 + a_1 \pi + a_2 \pi^2 + \cdots + a_k \pi^k + \cdots. \]
	Note that
	\[
	\begin{split}
		\varphi_k \bigl( \alpha \big(\tau(a_0)\big)^{-1} + \mathcal{ P}^k \bigr) 
		&= \varphi_k (\alpha +\mathcal{ P}^k) \cdot \varphi_k \big(\, \tau(a_0)^{-1} + \mathcal{ P}^k\big)\\
		&= (a_0 + \mathcal{ P}) \cdot (a_0 + \mathcal{ P})^{-1} \\
		&= 1 + \mathcal{ P},
	\end{split}
	\]
	thus \( \alpha (\tau(a_0))^{-1} + \mathcal{ P}^k \in U_1^{(k)}\). Therefore, the isomorphism in \eqref{2512231239} applies to \(\alpha\) as:
	\[
	\alpha +\mathcal{ P}^k = (\tau(a_0) +\mathcal{ P}^k) \bigl( \alpha (\tau(a_0))^{-1} +\mathcal{ P}^k \bigr) 
	\longmapsto (a_0 + \mathcal{ P}) \times \big(\alpha (\tau(a_0))^{-1}+ \mathcal{ P}^k\big).
	\]
	Hence we obtain the isomorphism
	\begin{equation}\label{2512241146}
		G_k = \langle \alpha +\mathcal{ P}^k \rangle \cong \langle a_0 + \mathcal{ P} \rangle \times \langle \alpha (\tau(a_0))^{-1} +\mathcal{ P}^k \rangle.
	\end{equation}
	Let $\eta=\alpha (\tau(a_0))^{-1}$, by \eqref{2512241146}, we have
	\begin{equation}\label{2512251206}
		\#\ke(G_k \to G_{k-1})=\#\ke\big(\langle\eta + \mathcal{ P}^k \rangle \to \langle\eta + \mathcal{ P}^{k-1} \rangle\big).
	\end{equation}
	
	We proceed to calculate the order of \(\eta + \mathcal{ P}^k\) in $U_1^{(k)}$, which is denoted by \(\operatorname{Ord}_k(\eta)\). By \cite[Theorem 80]{Hecke}, we obtain that  
	\[
	\begin{split}
		\#(\mathcal{O}_\mathfrak{p} / \mathcal{P}^k)^{\times}&= N(\mathfrak p^k)\left(1 - \frac{1}{N( \mathfrak{p})}\right)\\
		&= p^{f(k-1)}(p^f - 1),\quad \text{and}\\
		\#(\mathcal{O}_\mathfrak{p} / \mathcal{P})^{\times} &= p^f - 1,
	\end{split}
	\] 
	Combining this with \eqref{2512231239}, we have $\#U_1^{(k)} = p^{f(k-1)}$, hence \(\operatorname{Ord}_k(\eta)\) must be a power of \(p\). 
	
	Denote $v(j)=v_\mathfrak{p}(\eta^{p^j}-1)$ for each  $j\in\Z_{\ge0}$. Since $\eta^{p^{j}}\equiv 1\!\!\pmod{\mathcal{P}^{v(j)}}$ but $\eta^{p^{j}}\not\equiv 1\!\!\pmod{\mathcal{P}^{v(j)+1}}$, we can apply Lemma \ref{2601061529} with $a=\eta^{p^{j}}$ and $b=1$, and obtain
	\[
	v(j+1) \geq \min\{e + v(j), p v(j)\} > v(j).
	\]  
	Since \(\big(v(j)\big)_{j \geq 1}\) is strictly increasing, there exists \(l \in \Z_{\ge0}\) such that \(v(l) > e\).  
	When \(j \geq l\), since $pv(j)>e+v(j)$, we can apply \eqref{2512141206} in Lemma \ref{2601061529} with $a=\eta^{p^{j}}$ and $b=1$ to deduce
	\begin{equation}\label{2512251005}
		v(j+1)= v(j) + e.
	\end{equation}
	
	For each \( k > v(l) \), let \( \widetilde{k} = \left\lceil (k - v(l))e^{-1} \right\rceil \), where $\lceil x\rceil$ is the smallest integer not less than $x$. Using \eqref{2512251005} repeatedly, we have
	\[
	\begin{split}
		v(l+\tilde{k}) &= v(l) + e\widetilde{k}\\
		&\geq v(l) + (k - v(l)) = k.
	\end{split}
	\]
	Hence
	\(
	\eta^{p^{l+\widetilde{k}}} \equiv 1 \pmod{\mathcal{P}^k}
	\). Moreover, by using \eqref{2512251005} $\widetilde{k}-1$ times, we have
	\[
	\begin{split}
		v(l+\widetilde{k}-1) &= v(l) + e(\widetilde{k}-1)\\
		&< v(l) + k - v(l) = k,
	\end{split}
	\]
	thus $\eta^{p^{l+\widetilde{k}-1}} \not\equiv 1 \pmod{\mathcal{ P}^k}$. Therefore
	\begin{equation}
		\operatorname{Ord}_k(\eta)=p^{l+\widetilde{k}}.
	\end{equation}
	\begin{itemize}
		\item  If \( k - v(l) \equiv 1 \pmod{e} \), we can write \( k - v(l) = h e + 1 \), then
		\[
		\begin{split}
			\operatorname{Ord}_k(\eta) &= p^{l + \left\lceil \frac{h e + 1}{e} \right\rceil}
			= p^{l + h + 1},\\
			\operatorname{Ord}_{k-1}(\eta) &= p^{l + \left\lceil \frac{h e}{e} \right\rceil}
			= p^{l + h}.
		\end{split}
		\]
		\item  If \( k -v(l) \not\equiv 1 \pmod{e} \), we can write \( k - v(l) = h e + r \) with $r\in\{2,3,\dots,e\}$, then
		\[
		\begin{split}
			\operatorname{Ord}_k(\eta) &= p^{l + \left\lceil \frac{h e + r}{e} \right\rceil}
			= p^{l + h + 1},\\
			\operatorname{Ord}_{k-1}(\eta) &= p^{l + \left\lceil \frac{h e+r-1}{e} \right\rceil}
			= p^{l + h+1}.
		\end{split}
		\]
	\end{itemize}
	Combining this with \eqref{2512251206}, 
	\[
	\begin{split}
		&\#\ke(G_k\to G_{k-1})=\#\ke\bigl( \langle \eta + \mathcal{ P}^k \rangle \to \langle \eta + \mathcal{ P}^{k-1} \rangle \bigr) \\
		=& \frac{\#\langle\eta + \mathcal{ P}^k \rangle}{\#\langle  \eta + \mathcal{ P}^{k-1}\rangle}
		=\begin{cases}
			p,&\quad\text{if \( k - v(l) \equiv 1 \pmod{e} \)},\\
			1,&\quad\text{if \( k - v(l) \not\equiv 1 \pmod{e} \)},
		\end{cases}
	\end{split}
	\]
	which completes the proof of Theorem \ref{2505091458}.

	\section{Proof of Theorem \ref{2505041026} and Theorem \ref{2512261042}}
	
	Let $\alpha, \beta \in \mathcal{O}$ be as in Theorem \ref{2505041026} and recall that
	\[H_m = \langle \alpha+\beta^m\mathcal{O}\ \rangle \subseteq (\mathcal{O}/\beta^m\mathcal{O})^\times,\]
	is the multiplicative subgroup generated by $\alpha$ modulo $\beta^m\mathcal{O}$.    
	
	Since $(\beta)=\mathfrak{p}_1^{g_1}\cdots\mathfrak{p}_h^{g_h}$, by the Chinese remainder theorem, there exists an isomorphism  
	\[
	\alpha^n\!\! \!\! \pmod{\beta^m} \mapsto \left( \alpha^n \!\! \!\! \pmod{\mathfrak{p}_1^{g_1 m}}, \ldots, \alpha^n \!\! \!\! \pmod{\mathfrak{p}_h^{g_h m}} \right)
	\]
	such that
	\begin{equation}\label{2505301308}
		H_m \cong G_{g_1m}^{(1)} \times G_{g_2m}^{(2)} \times \cdots \times G_{g_hm}^{(h)},
	\end{equation}
	where $G_{k}^{(j)} := \langle \alpha+\mathfrak{ p}_j^{k} \rangle \subseteq (\mathcal{O}/\mathfrak{p}_j^{k})^\times$ for all $j\in\{1,2,\dots,h\}$. 
	
	\subsection{Proof of Theorem \ref{2505041026}}
	Since $\mathfrak{ p}_j$ is unramified for all $j\in\{1,2,\dots,h\}$, by Theorem \ref{2505091458}, there exists an positive integer $M$ such that
	\[\#G_{k}^{(j)}=\#\ke(G_{k}^{(j)}\to G_{k-1}^{(j)})\cdot\#G_{k-1}^{(j)}=p_j\cdot\#G_{k-1}^{(j)},\]
	for all $k\ge M$ and $j\in\{1,2,\dots,h\}$, where $p_j$ is the rational prime lying below $\mathfrak{p}_j$. Combining this with 
	\eqref{2505301308}, we have
	\begin{equation}\label{2505301327}
		\begin{split}
			h_m&:=\# H_m = \prod_{j=1}^h \# G_{g_jm}^{(j)} = \prod_{j=1}^h \left(p_j^{g_j} \# G_{g_j(m-1)}^{(j)}\right)\\
			& =\prod_{j=1}^h N(\mathfrak{ p}_j)^{g_j}\cdot\prod_{j=1}^h\# G_{g_j(m-1)}^{(j)}= N(\beta) \cdot\# H_{m-1}=N(\beta) h_{m-1},
		\end{split}
	\end{equation}
	for all $m\ge M+1$. 
	
	Recall that $(\alpha^n)_{\beta,i}$ is the $i$-th digit in the $\beta$-adic expansion of $\alpha^n$, we define 
	\[ D_{n,m}(b):=\#\{0 \leq i < m :  (\alpha^n)_{\beta,i} = b\},\]
	and 
	\[D_{m}(b):=\sum_{n=1}^{h_m} D_{n,m}(b),\]
	for each digit $b\in\mathcal{D}$. 
	
	Denote $(a_0, a_1,\cdots, a_{m-1})_{\beta}:=\sum_{j=0}^{m-1}a_j\beta^j$ for $(a_0, a_1,\cdots, a_{m-1})\in\mathcal{D}^m$, we have
	\[
	\begin{split}
		D_{m}(b)& = \sum_{n=1}^{h_m} \#\{0 \leq i < m :  (\alpha^n)_{\beta,i} = b\}\\
		&= \sum_{(a_0, a_1,\cdots, a_{m-1})_{\beta} \in H_m} \#\{0 \leq i < m :  a_i = b\}\\
		&=\sum_{(a_0, a_1,\cdots, a_{m-2})_{\beta} \in H_{m-1}}\;\sum_{a_{m-1} \in \mathcal D} \#\{ 0 \leq i < m : a_i = b\}, 
	\end{split}
	\]
	where we use \eqref{2505301327} in the last equality. Note that if $a_{m-1}=b$, then $$\#\{ 0 \leq i < m : a_i = b\}=1+\#\{ 0 \leq i < m-1 : a_i = b\},$$
	otherwise, $$\#\{ 0 \leq i < m : a_i = b\}=\#\{ 0 \leq i < m-1 : a_i = b\}.$$ Thus 
	\[
	\begin{split}
		D_{m}(b)&=\sum_{(a_0, a_1,\cdots, a_{m-2})_{\beta} \in H_{m-1}}(N(\beta)\cdot  \#\{ 0 \leq i < m-1 : a_i = b\}+1)\\
		&= N(\beta) D_{m-1}(b) + h_{m-1}.
	\end{split}
	\]
	Repeating the process above, we have
	\begin{equation}\label{2601061745}
		\begin{split}
			D_{m}(b)
			&= \big(N(\beta)\big)^2  D_{m-2}(b) + N(\beta) h_{m-2} + h_{m-1}\\
			&\vdots \\
			&= (N(\beta))^{m-M}  D_{M}(b) + (N(\beta))^{m-M-1} h_{M} + \cdots + h_{m-1}\\
			&= (N(\beta))^{m-M} D_{M}(b) + (m - M) (N(\beta))^{m-M-1} h_{M},
		\end{split}
	\end{equation}
	where we use \eqref{2505301327} in the last equality.
	Therefore, by \eqref{2601061743}, \eqref{2601061744}, \eqref{2601061745} and 
	\[ h_m = (N(\beta))^{m-M} h_{M},\]
	we have
	\[
	\begin{split}
		f_{\alpha,m}(b)&=\frac{1}{h_m}\sum_{n=1}^{h_m}f_{\alpha,n,m}(b)\\
		&=\frac{D_{m}(b)}{mh_m} = \frac{D_{M}(b)}{mh_M} + \frac{m - M}{m} (N(\beta))^{-1}.
	\end{split} 
	\]  Letting \( m \to \infty \), we obtain 
	\[\lim_{m \to \infty}f_m(b) = (N(\beta))^{-1} = 1/\#\mathcal{D},\] which completes the proof of Theorem \ref{2505041026}.
	
	\subsection{Proof of Theorem \ref{2512261042}}
	Since the ramification index of $\mathfrak{ p}_j$ is $e_j$, by Theorem \ref{2505091458}, there exists an positive integer $M$ such that for all $k\ge M$ and $j\in\{1,2,\dots,h\}$,
	\begin{align}
		\#G_{k}^{(j)}&=\#\ke(G_{k}^{(j)}\to G_{k-1}^{(j)})\cdot\#G_{k-1}^{(j)}  \notag\\
		&=\#\ke(G_{k}^{(j)}\to G_{k-1}^{(j)})\cdot\#\ke(G_{k-1}^{(j)}\to G_{k-2}^{(j)})\cdot\#G_{k-2}^{(j)}\notag\\
		&=\cdots=p_j\cdot\#G_{k-e_j}^{(j)},\label{2512261132}
	\end{align}
	where $p_j$ is the rational prime lying below $\mathfrak{p}_j$. Applying \eqref{2512261132} repeatedly, we have
	\begin{equation}
		p_j^{\big[\frac{k-M}{e_j}\big]}\#G_{M}^{(j)}\le\#G_{k}^{(j)}\le p_j^{\big[\frac{k-M}{e_j}\big]+1}\#G_{M}^{(j)}.
	\end{equation}
	Combining this with 
	\eqref{2505301308}, we have
	\begin{equation}
		\begin{split}
			\mathcal{C}_m(\alpha)&=\# H_m = \prod_{j=1}^h \# G_{g_jm}^{(j)} \le\prod_{j=1}^h \Big(p_j^{\big[\frac{g_jm-M}{e_j}\big]+1}\#G_{M}^{(j)}\Big),\;\;\text{ and }\\
			\mathcal{C}_m(\alpha)& \ge\prod_{j=1}^h \Big(p_j^{\big[\frac{g_jm-M}{e_j}\big]}\#G_{M}^{(j)}\Big),
		\end{split}
	\end{equation}
	for all $m\ge M$. Hence, setting $C_M:=\log(\prod_{j=1}^h\#G_{M}^{(j)})$, 
	\[\begin{split}
		&\frac{C_M+\sum_{j=1}^{h}(\frac{g_jm-M}{e_j}-1)\log p_j}{m\sum_{j=1}^{h}f_jg_j\log p_j}
		\le\frac{\log\mathcal{C}_m(\alpha)}{m\log N(\beta)}\\
		\le&\;\frac{C_M+\sum_{j=1}^{h}(\frac{g_jm-M}{e_j}+1)\log p_j}{m\sum_{j=1}^{h}f_jg_j\log p_j},
	\end{split}\]
	where we use $N(\beta)=\prod_{j=1}^{h}N(\mathfrak{ p})^{g_j}=\prod_{j=1}^{h}p_j^{f_jg_j}$. Letting \( m \to \infty \), we obtain \eqref{2512261352}, which completes the proof of Theorem \ref{2512261042}.
	
	
	
	\section*{Acknowledgements}
	R.~Li was supported by NSFC No.~12401006 and Guangdong Basic and
	Applied Basic Research Foundation No.~2023A1515110272. J.~Zhao was supported by NSFC No.~12471085 and Science and Technology Commission of Shanghai Municipality (STCSM) No. 22DZ2229014.

\end{document}